\documentclass{amsart}[12pt]
\def\doctype{}

\usepackage{latexsym,amssymb}
\usepackage{fancyhdr}
\usepackage{color}
\usepackage{tikz}
\usepackage{hyperref}

\newcommand\lam{\lambda}

\newcommand{\cB}{\mathcal{B}}

\newcommand\F{\mathbb{F}}

\newcommand{\comment}[1]{}

\numberwithin{equation}{section}

\fancypagestyle{titlepage}{
\fancyhead[R]{\doctype}
\fancyhead[CL]{}
\cfoot{\vspace{5pt} \thepage}
}

\let\oldsection\section
\newcommand\boldsection[1]{\oldsection{\bf #1}}
\newcommand\starsection[1]{\oldsection*{\bf #1}}
\makeatletter
\renewcommand\section{\@ifstar\starsection\boldsection}
\makeatother

\newtheoremstyle{theorem}
  {12pt}		  
  {0pt}  
  {\sl}  
  {\parindent}     
  {\bf}  
  {. }    
  { }    
  {}     
\theoremstyle{theorem}
\newtheorem{thm}{Theorem}[section]  
\newtheorem{lemma}[thm]{Lemma}     
\newtheorem{cor}[thm]{Corollary}

\newtheorem{prop}[thm]{Proposition}

\newtheoremstyle{definition}
  {12pt}		  
  {0pt}  
  {}  
  {\parindent}     
  {\bf}  
  {. }    
  { }    
  {}     
\theoremstyle{definition}

\newtheorem{ex}[thm]{Example}

\renewcommand{\proofname}{Proof}

\makeatletter
\renewenvironment{proof}[1][\proofname]{\par
  \pushQED{\qed}%
  \normalfont \partopsep=\z@skip \topsep=\z@skip
  \trivlist
  \item[\hskip\labelsep
        \scshape
    #1\@addpunct{.}]\ignorespaces
}{%
  \popQED\endtrivlist\@endpefalse
}
\makeatother

\makeatletter
\renewcommand*\@maketitle{%
  \normalfont\normalsize
  \@adminfootnotes
  \@mkboth{\@nx\shortauthors}{\@nx\shorttitle}%
  \global\topskip42\p@\relax 
  \@settitle
  \ifx\@empty\authors \else {\vskip 1em
\vtop{\centering\shortauthors\@@par}} \fi
  \ifx\@empty\@date \else {\vskip 1em \vtop{\centering\@date\@@par}}\fi 
  \ifx\@empty\@dedicatory
  \else
    \baselineskip18\p@
    \vtop{\centering{\footnotesize\itshape\@dedicatory\@@par}%
      \global\dimen@i\prevdepth}\prevdepth\dimen@i
  \fi
  \@setabstract
  \normalsize
  \if@titlepage
    \newpage
  \else
    \dimen@34\p@ \advance\dimen@-\baselineskip
    \vskip\dimen@\relax
  \fi
} 
\renewcommand*\@adminfootnotes{%
  \let\@makefnmark\relax  \let\@thefnmark\relax
  \ifx\@empty\@subjclass\else \@footnotetext{\@setsubjclass}\fi
  \ifx\@empty\@keywords\else \@footnotetext{\@setkeywords}\fi
  \ifx\@empty\thankses\else \@footnotetext{%
    \def\par{\let\par\@par}\@setthanks}%
  \fi
\thispagestyle{titlepage}
}
\makeatother

\begin{document}

\title{\large A lower bound on HMOLS with equal sized holes}

\author{Michael Bailey}
\address{\rm Michael Bailey:
Mathematics and Statistics,
University of Victoria, Victoria, BC, Canada
}
\email{mike.bailey122@gmail.com}

\author{Coen del Valle}
\address{\rm Coen del Valle: 
Mathematics and Statistics,
University of Victoria, Victoria, BC, Canada}
\email{cdelvalle@uvic.ca}

\author{Peter J.~Dukes}
\address{\rm Peter J.~ Dukes:
Mathematics and Statistics,
University of Victoria, Victoria, BC, Canada
}
\email{dukes@uvic.ca}

\thanks{Research of Peter Dukes is supported by NSERC grant 312595--2017}

\date{\today}

\begin{abstract}
It is known that $N(n)$, the maximum number of mutually orthogonal latin squares of order $n$, satisfies the lower bound $N(n) \ge n^{1/14.8}$ for large $n$.  For $h\ge 2$, relatively little is known about the quantity $N(h^n)$, which denotes the maximum number of `HMOLS' or mutually orthogonal latin squares having a common equipartition into $n$ holes of a fixed size $h$. 
We generalize a difference matrix method that had been used previously for explicit constructions of HMOLS.
An estimate of R.M. Wilson on higher cyclotomic numbers guarantees our construction succeeds in suitably large finite fields.
Feeding this into a generalized product construction, 
we are able to establish the lower bound $N(h^n) \ge (\log n)^{1/\delta}$ for any $\delta>2$ and all $n > n_0(h,\delta)$.
\end{abstract}

\maketitle
\hrule

\section{Introduction}

\subsection{Overview}

A \emph{latin square} is an $n \times n$ array with entries from an
$n$-element set of symbols such that every row and column is a permutation
of the symbols.  Often the symbols are taken to be from
$[n]:=\{1,\dots,n\}$.  The integer $n$ is called the \emph{order} of the
square.

Two latin squares $L$ and $L'$ of order $n$ are \emph{orthogonal} if
$\{(L_{ij},L'_{ij}): i,j \in [n]\}=[n]^2$; that is, two
squares are orthogonal if, when superimposed, all ordered pairs of symbols
are distinct.  
A family of latin squares in which any pair are orthogonal is 
called a set of \emph{mutually orthogonal latin squares}, or `MOLS' for
short.  The maximum size of a set of MOLS of order $n$ is denoted $N(n)$.
It is easy to see that $N(n) \le n-1$ for $n>1$,  with equality if and only
if there exists a projective plane of order $n$.  Consequently, $N(q)=q-1$
for prime powers $q$.  Using a number sieve and some recursive constructions,
Beth showed \cite{Beth} (building on \cite{CES,WilsonMOLS}) that
$N(n) \ge n^{1/14.8}$ for large $n$.  In fact, by inspecting the sieve 
a little more closely, $14.8$ can be replaced by $14.7994$; we use this observation
later to keep certain bounds a little cleaner.

In this article, we are interested in a variant on MOLS.
An \emph{incomplete latin square} of order $n$ 
is an
$n \times n$ array $L=(L_{ij}: i,j \in [n])$ with entries either blank or in $[n]$,
together with a partition $(H_1,\dots,H_m)$ of some subset of $[n]$ such that
\begin{itemize}
\item
$L_{ij}$ is empty if $(i,j) \in \cup_{k=1}^m H_k \times H_k$ and otherwise contains exactly one symbol;
\item
every row and every column in $L$ contains each symbol at most once; and
\item
symbols in $H_k$ do not appear in rows or columns indexed by $H_k$, $k=1,\dots,m$.
\end{itemize}
The sets $H_k$ are often taken to be intervals of consecutive
rows/columns/symbols (but need not be).
As one special case, when $m=n$ and each $H_k=\{k\}$, the definition is equivalent to a latin square $L$ that is \emph{idempotent}, that is satisfying $L_{ii}=i$ for each $i \in [n]$, except that the diagonal is removed to produce the corresponding incomplete latin square.

The \emph{type} of an incomplete latin square is the list $(h_1,\dots,h_m)$, where $h_i = |H_i|$ for each $i=1,\dots,m$.  When $h_i=n/m$ for all $k$, so that the set of holes is a uniform partition of $[n]$, the term `holey latin square' is used, and the type is abbreviated to $h^m$, where $h=n/m$.  To clarify the notation, we henceforth recycle the parameter $n$ as the number of holes, so that type $h^n$ is considered.  The relevant squares are then $hn \times hn$.

Two holey latin squares $L,L'$ of type $h^n$ (and sharing the same hole partition) are said to be \emph{orthogonal} if each of the $(hn)^2-nh^2$ ordered pairs of symbols from different holes appear exactly once when $L$ and $L'$ are superimposed.  As with MOLS, we use the term `mutually orthogonal' for a set of holey latin squares, any two of which are orthogonal.  The abbreviation HMOLS is standard in the more modern literature; see for instance \cite{ABG,Handbook}.  Following \cite[\S III.4.4]{Handbook}, we use a similar function $N(h^n)$ as for MOLS to denote the maximum number of HMOLS of type $h^n$.  (Some context is needed to properly parse this notation and not mistake `$h^n$' for exponentiation of integers.)

\begin{ex} As an example we give a pair of HMOLS of type $2^4$, also shown in ~\cite{DS}.  In our (slightly different) presentation, the holes are $\{1,2\}$, $\{3,4\}$, $\{5,6\}$, $\{7,8\}$.
\begin{center}
\begin{tabular}{|cccccccc|}
\hline
&&8&6&3&7&4&5\\
&&5&7&8&4&6&3\\
7&6&&&1&8&5&2\\
5&8&&&7&2&1&6\\
4&7&2&8&&&3&1\\
8&3&7&1&&&2&4\\
3&5&6&2&4&1&&\\
6&4&1&5&2&3&&\\
\hline
\end{tabular}
\hspace{1.5cm}
\begin{tabular}{|cccccccc|}
\hline
&&5&7&8&4&6&3\\
&&8&6&3&7&4&5\\
5&8&&&7&2&1&6\\
7&6&&&1&8&5&2\\
8&3&7&1&&&2&4\\
4&7&2&8&&&3&1\\
6&4&1&5&2&3&&\\
3&5&6&2&4&1&&\\
\hline
\end{tabular}
\end{center}
\end{ex}

It is easy to see that $N(n-1) \le N(1^n) \le N(n)$, so Beth's result gives a lower bound on HMOLS in the special case $h=1$. 
Also, if there exist $k$ HMOLS of type $1^n$ and $k$ MOLS of order $h$, then $N(h^n) \ge k$ follows easily by a standard product construction.
However, very little else is known about HMOLS with holes of a fixed size greater than $1$. Some explicit results are known for a small number of squares.  Dinitz and Stinson showed \cite{DS} that $N(2^n) \ge 2$ for $n \ge 4$.  Stinson and Zhu \cite{SZ} extended this to $N(h^n) \ge 2$ for all $h \ge 2$, $n \ge 4$.  Bennett, Colbourn and Zhu \cite{BCZ} settled the case of three HMOLS with a handful of exceptions.  Abel, Bennett and Ge \cite{ABG} obtained several constructions of four, five or six HMOLS and produced a table of lower bounds on $N(h^n)$ for $h \le 20$ and $n \le 50$.  For $2 \le h \le 6$, the largest entry in this table is 7, due to Abel and Zhang in \cite{AZ}.

As a na\"{i}ve upper bound, we have $N(h^n) \le n-2$ from a similar argument as for the standard MOLS upper bound.  In more detail, we may permute symbols in a set of HMOLS so that the first row contains symbols $h+1,\dots,nh$, where symbols $1,\dots,h$ are missing and columns $1,\dots,h$ are blank.  
Consider the symbols that occur in entry $(h+1,1)$ among the family of HMOLS.  At most one element from each of the holes $H_3,\dots,H_n$ can appear.

Our main result is a general lower bound on the rate of growth of $N(h^n)$ for fixed $h$ and large $n$.

\begin{thm}
\label{main}
Let $h$ be a positive integer and $\epsilon>0$.  For $k>k_0(h,\epsilon)$, there exists a set of $k$ HMOLS of type $h^n$ for all $n \ge k^{(3+\epsilon)\omega(h)k^2}$, where $\omega(h)$ denotes the number of distinct prime factors of $h$.  
\end{thm}
 
To our knowledge, it has not even been stated that $N(h^n)$ tends to infinity, though by now this is implicit from some results on graph decompositions; see Section~\ref{sec:graph-decompositions}.   The bound of Theorem~\ref{main} is very weak, yet we are satisfied at present due to the apparent difficulty of obtaining direct constructions.  Unlike in the case of MOLS, or in the case $h=1$, there is no obvious `finite-geometric' object to get started for hole size $h \ge 2$.  Indeed, the bulk of our lower bound is needed  to get just a single finite field construction of $k$ HMOLS; see Section~\ref{sec:cyclotomic}.

\subsection{Related objects}

Let $n$ and $k$ be positive integers, where $k \ge 2$.
A \emph{transversal design} TD$(k,n)$ consists of an $nk$-element set of \emph{points} partitioned into $k$ \emph{groups}, each of size $n$, and equipped with a family of $n^2$ \emph{blocks} of size $k$ having the property that any two points in distinct groups appear together in exactly one block.
There exists a TD$(k,n)$ if and only if there exists a set of $k-2$ MOLS of  order $n$. This equivalence is seen by indexing groups of the partition by rows, columns, and symbols from each square.  Transversal designs are closely connected to orthogonal arrays.

As with MOLS, it is possible to extend the definition to include holes.  A \emph{holey transversal design} HTD$(k,h^n)$ is a $khn$-element set, say $[k] \times X$, where $X$ has cardinality $hn$ and an equipartition $(H_1,\dots,H_n)$ of \emph{holes} of size $h$, together with a collection of $h^2n(n-1)$ blocks that cover, exactly once each, every pair of elements $(i,x)$, $(j,y)$ in which $i \neq j$ and $x,y$ are in different holes.  Of course, one could extend the definition to allow holes of mixed sizes, but the uniform hole size case suffices for our purposes.

For a graph $G$ and positive integer $t$, let $G(t)$ denote the graph obtained by replacing every vertex of $G$ by an independent set of $t$ vertices, and replacing every edge of $G$ by a complete bipartite subgraph between corresponding $t$-sets.  In other words, $G(t)$ is the lexicographic graph product $G \cdot \overline{K_t}$.  Let us identify in the natural way the set of points of an HTD$(k,h^n)$ with vertices of the graph $K_k(h) \times K_n$. If we interpret the blocks of the transversal design as $k$-cliques on the underlying set of points, then the condition that two elements appear in a block (exactly once) if and only if they are in distinct groups and distinct holes amounts to every edge of $K_k(h) \times K_n$ falling into precisely one $k$-clique.

The following is a summary of the preceding equivalences.

\begin{prop}
Let $h,k,n$ be positive integers with $k \ge 2$.  The following are equivalent:
\vspace{-10pt}
\begin{itemize}
\item
the existence of a set of $k-2$ HMOLS of type $h^n$;
\item
the existence of a holey transversal design HTD$(k,h^n)$; and
\item
the existence of a $K_k$-decomposition of $K_k (h) \times K_n$.
\end{itemize}
\end{prop}

\subsection{Existence via graph decompositions}
\label{sec:graph-decompositions}

In \cite{BKLOT}, Barber, K\"uhn, Lo, Osthus and Taylor prove a powerful existence result on $K_k$-decompositions of `dense' $k$-partite graphs.  In a little more detail, let us call a $k$-partite graph $G$ \emph{balanced} if every partite set has the same cardinality and \emph{locally balanced} if every vertex has the same number of neighbors in each of the other partite sets.  The main result of \cite{BKLOT} assures that any balanced and locally balanced $k$-partite graph on $kn$ vertices has a $K_k$-decomposition if $n$ is sufficiently large and the minimum degree satisfies $\delta(G) > C(k) (k-1)n$.  Here, $C(k)$ is a constant less than one associated with the `fractional $K_k$-decomposition' threshold.

We remark that the preceding machinery is enough to guarantee, for fixed $k$ and $h$, the existence of an HTD$(k,h^n)$ for sufficiently large $n$, since $K_k(h) \times K_n \cong K_k \times K_n(h)$ is $k$-partite and $r$-regular, where $r=h(k-1)(n-1)=(k-1)hn - h(k-1)$.  In fact, even a slowly growing parameter $h$ (as a function of $n$) can be accommodated.  However, the result in \cite{BKLOT} makes no attempt to quantify how large $n$ must be for the decomposition.  Even in the case of the structured $k$-partite graph we are considering, it is likely hopeless to obtain a reasonable bound on $n$ by this method.

Separately, the theory \cite{DMW,LW} of `edge-colored graph decompositions' due to R.M. Wilson and others, can be applied to the setting of HMOLS.  To sketch the details, we fix $h$ and $k$ and consider the graph $K_k(h) \times K_n$ for large $n$.  From this, we set up a directed complete graph $K_n^*$ with $r=(kh)^2-kh^2$ edge-colors between two vertices.  Each color corresponds with an edge of the bipartite graph $K_k(h) \times K_2$ occurring between two of the $n$ vertices.  Let $\mathcal{H}$ denote the family of all $r$-edge-colored cliques $K_k$ which correspond to legal placements of a block in our TD.  We seek an $\mathcal{H}$-decomposition of $K_n^*$, and this is guaranteed for sufficiently large $n$ from \cite[Theorem 1.2]{LW}. (We omit several routine calculations needed to check the hypotheses.)

Wilson's approach makes it difficult to obtain reasonable bounds on $n$, although in this context it is worth mentioning the bounds of Y.~Chang \cite{Chang-BIBD2,Chang-TD} for block designs and transversal designs.

\subsection{Outline}  The outline of the rest of the paper is as follows.  In Section~\ref{sec:cyclotomic}, we obtain a direct construction of $k$ HMOLS of type $h^q$ for large prime powers $q$.  This finite field construction is inspired by a method in \cite{DS} that was applied for $k \le 6$.  Then, in Section~\ref{sec:constructions}, we adapt a product-style MOLS construction in \cite{WilsonMOLS} to the setting of HMOLS.  The proof of our main result, Theorem~\ref{main}, is completed in Section~\ref{sec:proof}.  We conclude with a discussion of a few next steps for research on HMOLS.

\section{A cyclotomic construction}
\label{sec:cyclotomic}

\subsection{Expanding transversal designs of higher index}

Let $\lambda$ be a positive integer.  We define a TD$_\lam(k,n)$ similarly as a TD$(k,n)$, except that any two points in distinct groups appear together in exactly $\lam$ blocks (and otherwise in zero blocks).  
The integer $\lam$ is called the \emph{index} of the transversal design.

The main idea in what follows is to expand a TD$_\lam(k,h)$, where $h$ is the desired hole size, into an HTD$(k,h^q)$ for suitable large prime powers $q$.  Unless $k$ is small relative to $h$, the input design for this construction may require large index $\lam$.  When $h$ is itself a prime power, a 
TD$_\lam(k,h)$ naturally arises from a linear algebraic construction.

\begin{prop}
\label{td-projection}
Let $h$ be a prime power and $d$ a positive integer with $k \le h^d$.  Then there exists a TD$_{h^{d-1}}(k,h)$.
\end{prop}

\begin{proof}
Let $H$ be a field of order $h$.  Our construction uses points $H \times H^d$, where groups are of the form $H \times \{v\}$, $v \in H^d$.  Consider the family of blocks
$$\mathcal{B} =\{ \{ (a+u\cdot v,v): v \in H^d \} : a \in H, u \in H^d \},$$
where $u \cdot v$ denotes the usual dot product in the vector space $H^d$.  
The family $\mathcal{B}$ can be viewed as the result of developing the subfamily 
$\mathcal{B}_0 =\{ \{ (u\cdot v,v): v \in H^d \} : u \in H^d \}$
additively under $H$.
Fix two elements $v_1 \neq v_2$ in $H^d$ and a `difference' $\delta \in H$. Then since $|\{u : u \cdot (v_1-v_2)=\delta\}|=h^{d-1}$, it follows that there are exactly $h^{d-1}$ elements in $\mathcal{B}_0$ which achieve difference $\delta$ across the groups indexed by $v_1$ and $v_2$.  Therefore, two points $(a_1,v_1)$, $(a_2,v_2) \in H \times H^d$ are together in exactly one translate of each of those blocks, where $a_1-a_2=\delta$.  We have shown that $\mathcal{B}$ produces a transversal design of index $h^{d-1}$ on the indicated points and group partition; the restriction to (any) $k$ groups produces the desired TD$_{h^{d-1}}(k,h)$.
\end{proof}

We now use a standard product construction to build higher index transversal designs for the case where $h$ has multiple distinct prime divisors.

\begin{prop}
If there exists both a TD$_{\lambda_1}(k,h_1)$ and a TD$_{\lambda_2}(k,h_2)$, then there exists a TD$_{\lambda_1\lambda_2}(k,h_1h_2)$.
\end{prop}

\begin{proof}
Take the given TD$_{\lambda_i}(k,h_i)$ on point set $[k]\times H_i$, $i=1,2$, where $[k]$ indexes the groups. We construct our TD$_{\lambda_1\lambda_2}(k,h_1h_2)$ on points $[k]\times H_1 \times H_2$. For each block $\beta$ of the TD$_{\lambda_1}(k,h_1)$, we put the blocks of a TD$_{\lambda_2}(k,h_2)$ on $\{(x,y,z) : (x,y)\in\beta, z\in H_2\}$. It is easy to verify the resulting design is a TD$_{\lambda_1\lambda_2}(k,h_1h_2)$.
\end{proof}

The next result follows immediately from the previous two propositions and induction.

\begin{cor}
\label{prod-inf}
Let $h \ge 2$ be an integer which factors into prime powers as  $q_1q_2\cdots q_{\omega(h)}$. 
Put $\lam(h,k):=\prod_{i=1}^{\omega(h)} q_i^{d_i-1}$, where $d_i=\lceil \log_{q_i} k\rceil$ for each $i$.
Then there exists a TD$_\lambda(k,h)$.
\end{cor}

%

Next, we show how to expand a transversal design of group size $h$ and index $\lam$ into an HTD with hole size $h$ (and index one). Roughly speaking, elements are expanded into copies of a finite field $\F_q$, where $q \equiv 1 \pmod{\lam}$.  Each block is lifted so that previously overlapping pairs now cover the cyclotomic classes of index $\lam$, and then blocks are developed additively in $\F_q$.  To this end, we cite a guarantee of R.M. Wilson on cyclotomic difference families in sufficiently large finite fields.

\begin{lemma}[Wilson; see \cite{WilsonCyc}, Theorem 3]
\label{wilson-cyc}
Let $\lambda$ and $k$ be given integers, $\lambda,k \ge 2$.  For any prime power $q \equiv 1 \pmod{\lambda}$ with $q>\lambda^{k(k-1)}$, there exists a $k$-tuple $(a_1,\dots,a_k) \in \F_q^k$ such that the $\binom{k}{2}$ differences $a_j-a_i$, $1 \le i < j \le k$, belong to any prespecificed cosets of the index-$\lambda$ subgroup of $\F_q^\times$.
\end{lemma}

Applying this, we have the following result which mirrors \cite[Construction 6]{DLL}.

\begin{prop}
\label{htd-cons}
Suppose there exists a TD$_\lam(k,h)$ and $q$ is a prime power with $q \equiv 1 \pmod{\lam}$, $q>\lam^{k(k-1)}$.  Then there exists an HTD$(k,h^q)$.
\end{prop}

\begin{proof}
Consider a TD$_\lam(k,h)$ on $[k] \times H$, with block collection $\mathcal{B}$. Consider the collection of point-block incidences $S:=\{((x,y),\beta) : (x,y)\in\beta\in\mathcal{B}\}$.  Let $\mu : {S \choose 2}\to\{0,1,\dots,\lambda-1\}$ be defined such that for each fixed pair $(i,y), (j,y')$ with $1\leq i<j\leq k$, $$\left\{\mu(\{((i,y),\beta),((j,y'),\beta)\right\}) : \beta\supset \{(i,y),(j,y')\}\}=\{0,1,\dots,\lambda-1\}.$$  (One can choose such a $\mu$ via a `greedy 
labeling'.)
Pick a prime power $q\equiv 1\pmod{\lambda}$, $q>\lam^{k(k-1)}$, and let $C_0,C_1,\dots, C_{\lambda-1}$ denote the cyclotomic classes of index $\lambda$ in $\F_q$. By Lemma~\ref{wilson-cyc} there is a map $\phi : S\to \F_q$ such that for every block $\beta\in\mathcal{B}$, and $i<j$, $\phi((i,y),\beta)-\phi((j,y'),\beta)\in C_t$, where $t=\mu(\{((i,y),\beta),((j,y'),\beta)\})$. We construct an HTD$(k,h^q)$ on  $[k]\times H\times \F_q$ as follows. For each $a\in C_0$, $\beta\in \mathcal{B}$, and $c\in \F_q$, include the block $a\beta'+c$, where $a\beta'+c=\{(x,y,a\phi((x,y),\beta)+c) : (x,y)\in \beta\}$.

It is clear that if two points are in the same group, they will appear together in no common blocks; this is inherited from the original TD$_\lam(h,k)$.  Consider two points in different groups, but the same hole, say $(x,y,z)$, and $(x',y',z)$, where $x \neq x'$. If there were some block $a\beta'+c$ containing both points then we would have $\phi((x,y),\beta)=\phi((x',y'),\beta)$, an impossibility.  It remains to show that any two points from different groups and holes appear together in exactly one block. Let $(x,y,z)$, and $(x',y',z')$ be two such points. By construction there is exactly one block $\beta$ satisfying $z-z'\in C_t$ where $t=\mu(\{((x,y),\beta),((x',y'),\beta)\})$.  Then, there is some $a\in C_0$ satisfying $a(\phi((x,y),\beta)-\phi((x',y'),\beta))=z-z'$, and so our two chosen points belong to the block $a\beta'+c$, where $c=z-a\phi((x,y),\beta)$.
\end{proof}

Combining Corollary \ref{prod-inf} and Proposition \ref{htd-cons}, we obtain a construction of HTD$(k,h^q)$ for general $h$ and $k$ and certain large prime powers $q$.

\begin{thm}
\label{cyclotomic-hmols}
Let $h \ge 2$ be an integer which factors into prime powers as $q_1q_2\cdots q_{\omega(h)}$. 
Then there exists an HTD$(k,h^q)$ for all prime powers $q\equiv 1\pmod{\lam(h,k)}$, $q>\lam(h,k)^{k(k-1)}$.  In other words, $N(h^q) \ge k$ for all prime powers $q\equiv 1 \pmod{\lam(h,k+2)}$ with
$q>\lam(h,k+2)^{(k+2)(k+1)}$.
\end{thm}


\subsection{Template matrices and explicit computation}

We include here some remarks on explicit computer-aided construction of HTD$(k,h^q)$ in the special case of prime hole size $h$.
The `expansion' construction of Proposition~\ref{htd-cons} relies on lifting all blocks so that the differences across any two points fall into distinct cyclotomic classes.  Using a `template matrix' method introduced by Dinitz and Stinson \cite{DS}, it is possible to impose some additional structure on this lifting to gain an efficiency in computations.

With notation similar to before, we define an $h^d \times h^d$ `template matrix' $T_d(h)$ as the Gram matrix of the vector space $H^d$.  That is, rows and columns of $T_d(h)$ are indexed by $H^d$, and $T_d(h)_{uv} = u \cdot v$.  Although the order in which columns appear is unimportant, it is convenient to index the rows in lexicographic order.  When $h=2$, the template is simply a Walsh Hadamard matrix (with entries $0,1$ instead of $\pm 1$).  We offer another example below.

\begin{ex}
Consider the case $h=3$, $d=2$, which is suitable for the construction of up to $6=2^3-2$ HMOLS having hole size $3$.
With rows (and columns) indexed by the lex order on $\F_3^2$, we have
$$T_3(2)=
\left[\begin{array}{ccc|ccc|ccc}
0&0&0&0&0&0&0&0&0\\
0&1&2&0&1&2&0&1&2\\
0&2&1&0&2&1&0&2&1\\
\hline
0&0&0&1&1&1&2&2&2\\
0&1&2&1&2&0&2&0&1\\
0&2&1&1&0&2&2&1&0\\
\hline
0&0&0&2&2&2&1&1&1\\
0&1&2&2&0&1&1&2&0\\
0&2&1&2&1&0&1&0&2\\
\end{array}\right].$$
\end{ex}

Observe that the difference between any two distinct columns of $T_h(d)$ achieves every value in $H$ exactly $h^{d-1}$ times each; this is essentially the content of Proposition~\ref{td-projection}.  For $k \le h^d$, the restriction of $T_h(d)$ to any $k$ columns has the same property.  As we illustrate in Example~\ref{401} to follow, aiming for a value of $k$ less than $h^d$ may be a worthwhile tradeoff in computations.

We use the template matrix in conjunction with the following `relative difference matrix' setup for HTDs.

\begin{lemma}[see \cite{DS}]\label{difcon}
Let $G$ be an abelian group of order $g$ with subgroup $H$ of order $h$, and $\cB\subseteq G^k$. If for all $r,s$ with $1\leq r<s\leq k$ and each $a\in G\setminus H$ there is a unique $b\in \cB$ with $b_r-b_s=a$, then there exists an HTD$(k,h^{g/h})$. 
\end{lemma}

To further set up the construction, fix integers $d,h \ge 2$, and put $\lam=h^{d-1}$.  Let $q\equiv 1\pmod{\lam}$ be a prime power.
Let $\omega$ be a multiplicative generator of $\F_q$, and define $C_0:= \langle \omega^{\lam} \rangle$
to be the index-$\lam$ subgroup of $\F_q^\times$.   
For $1\leq i < \lam$, we denote the coset $\omega^i C_0$ by $C_i$.
Let $k \le h^d$.
Given two $k$-tuples $t \in X^k$ and $u \in Y^k$, define $t \circ u=((t_i,u_i) : 1 \le i \le k) \in (X \times Y)^k$.  We take $X=H=\F_h$ and $Y=F_q$ in what follows, so that the group in Lemma~\ref{difcon} is $G= \F_h \times \F_q$ with subgroup $\F_h \times \{0\}$.

Letting $t_1,t_2,\dots, t_{h^d}$ denote the rows of $T_h(d)$, our construction amounts to a selection of vectors $u_1,u_2,\dots,u_{h^d} \in \F_q^k$ such that 
$$\mathcal{B} = \{ t_i \circ (x u_i) : x \in C_0, 1 \le i \le h^d \}$$
satisfies the hypotheses of Lemma~\ref{difcon}.  
Dinitz and Stinson \cite{DS} and later, Abel and Zhang \cite{AZ}, reduce the search for such vectors $u_i$ by assuming they have the form 
$$
u_1, \omega u_1, \omega^2 u_1, \dots, \omega^{\lam-1}u_1, \dots, 
u_h, \omega u_h, \omega^2 u_h, \dots, \omega^{\lam-1}u_h.$$
With this reduction, $\mathcal{B}$ produces an HTD if the quotients $(u_{ir}-u_{is})(u_{jr}-u_{js})^{-1}$ lie in certain cosets of $C_0$ for each pair $r,s$ with $1 \le r < s \le k$. In more detail, fix two such column indices and consider two blocks $b,b' \in \cB$ arising from a choice of two rows of $T_d(h)$.  When these rows are in the same block of $\lam=h^{d-1}$ consecutive rows, we automatically avoid $b_r-b_s = b'_r-b'_s$ because of the different powers of $\omega$ multiplying the same $u_i$.  On the other hand, when these rows are in, say, the $i$th block and $j$th block of $\lam$ rows, $i\neq j$, we must ensure that the quotient  $(u_{ir}-u_{is})(u_{jr}-u_{js})^{-1}$ avoids those cyclotomic classes indexed by $e'-e \pmod{\lam}$, whenever $\omega^e u_i$ and $\omega^{e'} u_j$ index rows of $T_d(h)$ which have equal $(r,s)$-differences.
It is routine (but somewhat tedious) exercise to characterize the `allowed cosets', either computationally for specific $h,d$ or in general.  We omit the details, but point out that, for each $r$ and $s$, an arithmetic progression of cyclotomic classes (with difference a power of $h$) is available.

Now, given a table of allowed cosets, the vectors $u_1,\dots,u_h \in \F_q^k$ can be chosen one at a time, where each new vector has coset restrictions on its $(r,s)$-differences.  The guarantee of Lemma~\ref{wilson-cyc} can be used for this purpose (giving an alternate proof of Corollary~\ref{cyclotomic-hmols} in the case of prime $h$).  However, in practice it often suffices to take significantly smaller values of $q$.

\begin{ex}
\label{401}
To illustrate the method, we construct 9 HMOLS of type $2^{401}$ in $\F_2 \times \F_{401}$; that is, we consider $h=2$, $q=401$.  Instead of using all $16$ columns of the template $T_2(4)$, we require only $9+2=11$, as indicated below.  Let
\begin{align*}
u_1&=(284, 136, 249, 334, 1, 202, 140, 307, -, 35, 312, -, 0, -, -, -)\\
\text{and}~u_2&=(283, 297, 137, 60, 1, 210, 102, 39, -, 241, 111,-, 0, -, -, -).
\end{align*}
It can be verified that the quotients $(u_{2r}-u_{2s})(u_{1r}-u_{1s})^{-1}$ all lie in allowed cosets for $T_2(4)$ for any distinct indices $r$ and $s$ such that our vectors are nonblank.  (As an explanation for the unnatural ordering of entries, it turns out that a column-permutation of the template $T_2(4)$ was more convenient for the computations, at least with our approach.)
\end{ex}

To our knowledge, Example~\ref{401} provides the first (explicit) construction of more than $6$ HMOLS of type $2^n$ for any $n>1$.

\section{Recursive constructions}
\label{sec:constructions}

As we move away from prime powers, we present a product construction which scales the number of holes.  The idea is to join (copies of) equal-sized HMOLS on the diagonal and ordinary MOLS off the diagonal.  Our proof uses the language of transversal designs.

\begin{prop}
\label{prod2}
$N(h^{mn}) \ge \min \{N(1^m),N(hn),N(h^n) \}$.
\end{prop}

\begin{proof}
We show that the existence of a HTD$(k,h^{mn})$ is implied by the existence of an HTD$(k,1^m)$, TD$(k,hn)$ and HTD$(k,h^n)$.  Let us take as points $[k] \times H \times M \times X$, where $|H|=h$, $|M|=m$, $|X|=n$.  The groups of our TD are $\{i\} \times H \times M \times X$ and the holes are $[k] \times H \times \{w\} \times \{x\}$, where $i \in [k]$, $w \in M$, and $x \in X$.  We construct the block set in two pieces.  First, on each layer of points of the form $[k] \times H \times \{w\} \times X$, where $w \in M$, we include the blocks of an HTD$(k,h^n)$ with groups $\{i\} \times H \times \{w\} \times X$ and holes $[k] \times H \times \{w\} \times \{x\}$. Second, let us take an HTD$(k,1^m)$ on $[k] \times M$ and, for each block $B=\{(i,w_i): i = 1,\dots,k\}$, include the blocks of a TD$(k,hn)$ on $\cup_i \{i\} \times H \times \{w_i\} \times X$, where in each case we use the natural group partition induced by first coordinates.

It remains to verify that the block set as constructed covers every pair of points as needed for an HTD$(k,h^{mn})$.  To begin, since each of our ingredient blocks is transverse to the group partition, it is clear that two distinct points in the same group are together in no block.  Moreover, two distinct points in the same hole appear in the same HTD$(k,h^n)$, the hole partition of which is inherited from our resultant design.  Therefore, such elements are also together in no block.
Consider then, two elements $(i_1,j_1,w_1,x_1)$ and $(i_2,j_2,w_2,x_2)$ with $i_1 \neq i_2$ and $(w_1,x_1) \neq (w_2,x_2)$.  If $w_1=w_2$, this pair of points occurs in the same HTD$(k,h^n)$, and thus in exactly one block. On the other hand, if $w_1 \neq w_2$, we first locate the unique block of the HTD$(k,1^m)$ containing $(i_1,w_1)$ and $(i_2,w_2)$, and then, within the corresponding TD$(k,hn)$, identify the unique block containing our two given points.  We remark that the holes can be safely ignored in this latter case, since we are assuming $w_1 \neq w_2$.
\end{proof}

Although the idea behind the construction in Proposition~\ref{prod2} is very standard, we could not find this result mentioned explicitly in the literature.

To set up our next construction, we recall that an incomplete latin square can have holes that partition a proper subset of the index set.  In particular, we consider sets of mutually orthogonal $n \times n$ incomplete latin squares with a single common $h \times h$ hole for integers $n>h>0$.  The maximum number of squares in such a set is commonly  denoted $N(n;h)$.

\begin{ex}
A noteworthy value is $N(6;2)=2$ in spite of the nonexistence of a pair of orthogonal latin squares of order six.  The squares, with common hole $H=\{1,2\}$, are shown below.
\begin{center}
\begin{tabular}{|cccccc|}
\hline
&&3&4&5&6\\
&&4&3&6&5\\
6&3&5&1&4&2\\
4&5&6&2&3&1\\
3&6&2&5&1&4\\
5&4&1&6&2&3\\
\hline
\end{tabular}
\hspace{1.5cm}
\begin{tabular}{|cccccc|}
\hline
&&3&5&6&4\\
&&4&6&5&3\\
3&5&2&4&1&6\\
6&4&1&3&2&5\\
4&6&5&1&3&2\\
5&3&6&2&4&1\\
\hline
\end{tabular}
\end{center}
\end{ex}

Given a set of $k-2$ mutually orthogonal incomplete latin squares of type $(n;h)$, reading blocks as $k$-tuples produces 
an \emph{incomplete transversal design}, abbreviated either ITD$(k,(n;h))$ or $\text{TD}(k,n)-\text{TD}(k,h)$. The latter notation is not meant to suggest  that a TD$(k,h)$ exists as a subdesign, but rather that two elements from the hole are uncovered by blocks.
The interested reader is referred to \cite{Handbook,MOLStable} for more information and references on these objects.

We now extend Proposition~\ref{prod2} to get an analog of Wilson's MOLS construction \cite[Theorem 2.3]{WilsonMOLS}.  

\begin{prop}
\label{wilsonish-cons}
For $0 \le u < t$,
\begin{equation*}
N(h^{mt+u}) \ge \min \{N(t)-1,N(h^m),N(hm),N(hm+h;h),N(h^u) \}.
\end{equation*}
\end{prop}

\begin{proof}
We show that the existence of an HTD$(k,h^{mt+u})$ is implied by the existence of a TD$(k+1,t)$, HTD$(k,h^m)$, TD$(k,hm)$, HTD$(k,h^u)$ and an ITD$(k,(hm+h;h))$.
We remark that the first of these is equivalent to a resolvable TD$(k,t)$. 

The set of points for our design is $[k] \times H \times (M \times X \cup Y)$, where $|H|=h$, $|M|=m$, $|X|=t$, and $|Y|=u$.  Similar to the proof of Proposition~\ref{prod2}, the groups are induced by first coordinates, and the holes are `copies of $H$'.  

Begin with a TD$(k+1,t)$ on $([k] \cup \{0\}) \times X$, say with blocks $\mathcal{A}$.  Let $x_* \in X$ and assume, without loss of generality, that the blocks in $\mathcal{A}$ incident with $(0,x_*)$ are of the form $\{(i,x):i=1,\dots,k\} \cup \{(0,x_*)\}$.  In other words, in the induced resolvable TD$(k,t)$, assume one parallel class is labeled as $[k] \times \{x\}$, $x \in X$. Let us identify $Y$ with any $u$-element subset of $\{0\} \times (X \setminus \{x_*\})$.  

For each block in $\mathcal{A}$ of the form $\{(i,x):i=1,\dots,k\} \cup \{(0,x_*)\}$, include the blocks of an HTD$(k,h^m)$, on $[k] \times H \times M \times \{x\}$ with groups and holes as usual.  
Consider now a block $B \in \mathcal{A}$ which does not contain $(0,x_*)$, say $B=\{(i,x_i):i=0,1,\dots,k\} \in \mathcal{A}$ where $x_0 \neq x_*$.
Put $y_0 = (0,x_0)$.   If $y_0 \not\in Y$ (that is if $B$ does not intersect $Y$), we include the blocks of a TD$(k,hm)$ on $\cup_{i=1}^k \{i\} \times H \times M \times \{x_i\}$ with groups and holes as usual.  On the other hand, if $y_0 \in Y$ (that is if $B$ intersects $Y$), we include the blocks of an ITD$(k,(hm+h;h))$ on the points $\cup_{i=1}^k \{i\} \times H \times (M \times \{x_i\} \cup \{y_0\})$ and such that the hole of this ITD occurs as
$\cup_{i=1}^k \{i\} \times H \times  \{y_0\}$.  To finish the construction, we include the blocks of an HTD$(k,h^u)$ on $[k] \times H \times Y$, where again the natural partition into groups and holes is used.  We have used four types of blocks, to be referenced below in the order just described.

As a verification, we consider two elements in different groups and holes.   Suppose $i$ and $i'$ index two different groups.  There are cases to consider. 
Consider first a pair of points of the form $(i,j,w,x)$ and $(i',j',w',x')$, where $(w,x) \neq (w',x')$.  If $x=x'$, then the two points appear together in exactly one block of the first kind.  If $x \neq x'$, then we consider the unique block $B \in \mathcal{A}$ containing $(i,x)$ and $(i',x')$.  Our two points are either in exactly one block of the second kind if $B \cap Y = \emptyset$ or exactly one block of the third kind otherwise.
Consider now the points $(i,j,w,x)$ and $(i',j',y')$, where $y' \in Y$.  There is exactly one block of the TD$(k+1,t)$ containing $(i,x)$ and $y_0$.  Examining the ITD prescribed by the construction, the points $(i,j,w,x)$ and $(i',j',y')$ appear together in exactly one block of the third kind, since $i \neq i'$ and only one of these points belongs to the hole. 
Finally, two points $(i,j,y)$ and $(i',j',y')$  appear together in one (and only one) block of the fourth kind if and only if $y \neq y'$.
\end{proof}

\begin{rk}
The construction in Proposition~\ref{prod2} is just the specialization $Y=\emptyset$ of that of Proposition~\ref{wilsonish-cons}.  However, we have kept the former stated separately since it requires no assumption on $N(hm+h;h)$.
\end{rk}

\section{Lower bounds}
\label{sec:proof}

\subsection{Preliminary bounds}
To make use of Proposition~\ref{wilsonish-cons}, it is helpful to have a lower bound on $N(n;h)$ resembling Beth's bound for $N(n)$.

\begin{thm}
\label{imols-bound}
Let $h$ be a positive integer.  Then $N(n;h) > n^{1/29.6}$ for sufficiently large $n$. 
\end{thm}

\begin{proof}
We use the construction of \cite[Theorem 2.4]{WilsonMOLS}.  A minor variant gives that, for $0 \le u,v \le t$,
\begin{equation}
\label{imols-cons}
N(mt+u+v;v) \ge \min \{N(m),N(m+1),N(m+2),N(t)-2,N(u)\}.
\end{equation}
(To clarify, the cited theorem `fills the hole' of size $v$ so that the left side becomes $N(mt+u+v)$ and the minimum on the right side includes $N(v)$.)
Put $v=h$ and suppose $k$ is a large integer.  For $n \ge k^{29.6}$, write $n = mt+u$, where $m,t,u \ge k^{14.7995}$.  Then, from Beth's inequality, there exist $k$ MOLS of each of the side lengths $m,m+1,m+2,u$, and also $k+2$ MOLS of side length $t$.
It follows from (\ref{imols-cons}) that $N(n+h;h) \ge k$, as required.
\end{proof}

The forthcoming proof of Theorem~\ref{main} makes use of two number-theoretic lemmas which are minor variants of classical results.  The first of these concerns the selection of a prime, with a congruence restriction, in a large and wide enough interval.

\begin{lemma}
\label{Dirichlet-estimate}
For any sufficiently large integer $M$ and any real number $x>e^M$ there exists a prime $p \equiv 1 \pmod{M}$ satisfying $x < p \le 2x$.
\end{lemma}

\begin{proof}
We use a result \cite[Theorem 1.3]{Dirichlet-bounds} of Bennett, Martin, O'Bryant and Rechnitzer concerning the prime-counting function 
$$\pi(x;q,a) := \#\{p \le x: p~\text{ is prime}, p \equiv a \pmod{q}\}.$$
With $q=M$ and $a=1$, their estimate implies
\begin{equation*}
\left| \pi(x;M,1) - \frac{{\mathrm{Li}}(x)}{\phi(M)} \right| \le \frac{1}{160} \frac{x}{(\log x)^2}
\end{equation*}
for $M>10^5$ and all $x > e^M$. Since $\mathrm{Li}(x)\sim x/\log(x)$ and $\phi(M) < \log x$, a routine calculation gives $\pi(2x;M,1)-\pi(x;M,1) \ge 1$ for sufficiently large $x$ and $M$.
\end{proof}

\begin{rk}
Lemma~\ref{Dirichlet-estimate} actually holds with `2' replaced by any constant greater than one; however, the present form suffices for our purposes. 
\end{rk}

Next, we 
have a Frobenius-style representation theorem for large integers.

\begin{lemma}
\label{Frobenius}
Let $a,b$ and $C$ be positive integers with $\gcd(a,b)=1$.  Any $n > a(b+1)(b+C)$ can be written in the form $n=ax+by$ where $x$ and $y$ are integers satisfying $x \ge C$ and $y > ax$.
\end{lemma}

\begin{proof}
The integers $a(C+1), a(C+2), \dots, a(C+b)$ cover all congruence classes mod $b$.  Suppose $n \equiv a(C+j) \pmod{b}$, where $j \in \{1,\dots,b\}$.  Put $x=C+j$ and
$y=(n-ax)/b$.  Then $y$ is an integer with
\begin{equation*}
y > \frac{a(b+1)(C+b)-a(C+b)}{b} = a(C+b) \ge ax.\qedhere
\end{equation*}
\end{proof}

\subsection{Proof of the main result}

We are now ready to prove our asymptotic lower bound on HMOLS of type $h^n$.

\begin{proof}[Proof of Theorem~\ref{main}]
Put  $M=\lam(h,k+2)$, as defined in Corollary~\ref{prod-inf}. Note that $M \le (k+2)^{\omega(h)}$.  Let $K$ denote the ceiling of 
$k^{(\omega(h)+\epsilon/4)k^2}$, which we note for large $k$ exceeds both $e^M$ and $M^{(k+2)(k+1)}$.

Using Lemma~\ref{Dirichlet-estimate}, choose two primes $q_1,q_2 \equiv 1 \pmod{M}$ where $q_2 \in (K,2K]$ and $q_1 \in (2K,4K]$.
With $m=q_2$, we have $N(h^m) \ge k$ from Theorem~\ref{cyclotomic-hmols},  $N(hm) \ge k$ from Beth's inequality, and $N(hm+h;h) \ge k$ from Theorem \ref{imols-bound}.  The latter two bounds use the assumption that $k$ is large.

From the hypothesis on $n$ and choice of $q_i$, we have for large $k$,
$$n > k^{(3+\epsilon)\omega(h)k^2} > 17K^3 > q_1(q_2+1)(q_2+k^{14.8}).$$
Using Lemma~\ref{Frobenius},  write $n=q_1s+q_2t$, where $s,t$ are integers satisfying $s \ge k^{14.8}$ and $t > q_1s$.  Put $u=q_1s$ so that, with this alternate notation, we have $n=mt+u$ with $t>u$.  Observe that $N(h^u) \ge k$ from Proposition~\ref{prod2} with $s$ taking the role of $m$ and $q_1$ taking the role of $n$.  We additionally have $N(t) > k$ from Beth's inequality and our lower bound on $t$.

From the above properties of $m,t,u$, Proposition~\ref{wilsonish-cons} implies $N(h^n) \ge k$.
\end{proof}

\begin{ex}
We illustrate the proof method by computing an explicit  bound for the existence of six HMOLS of type $2^n$.  We show that $n> 8 \times 50 \times 148 = 59200$ suffices.

From \cite[Table 1]{ABG}, there exist six HMOLS of types $2^8$ and $2^{49}$.  (The former does not arise from the cyclotomic construction of Section 2, but it helps us optimize the bound.)  
From \cite[Table III.3.83]{Handbook}, we have $N(1^s) \ge 6$ for all $s \ge 99$.  It follows by Proposition~\ref{prod2} that $N(2^{8s}) \ge 6$ for all $s \ge 99$.  Put $m=49$ and note that
$N(2m)=N(98) \ge 6$ and $N(2m+2;2) = N(100;2) \ge 6$, where the latter appears in \cite[Table III.4.14]{Handbook}.
Write $n=8s+25t$ where $t >8s$.  From \cite[Table III.3.81]{Handbook}, we have $N(t) \ge 7$.  Letting  $h=2$ and $u=8s$, we conclude from Proposition~\ref{wilsonish-cons} that $N(2^n) \ge 6$.
\end{ex}

\subsection{Inverting the bound}

Here, we offer a lower bound on $N(h^n)$ in terms of $n$.

\begin{thm}\label{N-bound}
Let $h\geq 2$ be an integer and $\delta>2$ a real number.  Then $N(h^n)\ge (\log n)^{1/\delta}$ for all $n>n_0(h,\delta)$.
\end{thm}

\begin{proof}
If $k$ is an integer not exceeding the right side of the above bound, then $\log n \ge k^\delta > C k^2 \log k $ for any constant $C >3 \omega(h)$ and sufficiently large $k$.  The existence of $k$ HMOLS of type $h^n$ then follows from Theorem~\ref{main}.
\end{proof}

\begin{rk}
Various slightly better `inverse bounds' are possible.  For instance, we have
$N(h^n)\ge e^{\frac{1}{2}W(\log(n)/2\omega(h))}$ for sufficiently large $n$, where $W$ denotes Lambert's function, the inverse of $x \mapsto xe^x$.  The constants here represent one choice of many.
\end{rk}

\section{Future directions}

It would of course be desirable to produce a lower bound of the form $N(h^n) \ge n^{\delta}$ for some $\delta>0$.  This appears difficult using the presently available methods over finite fields, although a sophisticated randomized construction is plausible.

When $k$ is a fixed positive integer (rather than sufficiently large), our proof method can still compute, in principle, a lower bound on $n$ such that $N(h^n) \ge k$.  However, this bound incurs a considerable penalty in the analytic number theory for small integers $k$.
To track this penalty, we require a bound on the selection of prime powers in the spirit of \cite[Lemma 5.3]{Chang-BIBD1} or a deeper look at explicit estimates for primes in Dirichlet's theorem such as in the data attached to \cite{Dirichlet-bounds}.

Concerning explicit sets of HMOLS, we showed $N(2^{401}) \ge 9$ in Example~\ref{401}.  A few other sets of 9 HMOLS of type $2^n$ for $n<1000$ were found, along with a set of 10 HMOLS of type $2^{1009}$.  Our search for a pair of vectors with the needed cyclotomic constraints was very na\"ive.  We also restricted our computational efforts to the case $h=2$.
With some improved code to search for vectors $u_1,u_2,\dots,u_h$ producing a set of HMOLS, one could envision an expanded table of lower bounds.  Such an undertaking could offer a better sense of what to expect in practice from the standard construction methods and set a benchmark for future bounds.

As a separate line of investigation, it would be of interest to improve the exponent of Theorem~\ref{imols-bound}, our bound on MOLS with exactly one hole of a fixed size.  A direct use of the Buchstab sieve as in \cite{WilsonMOLS} is likely to do a bit better than our (indirect) method.

\end{document}